\documentclass[11pt]{article}

\usepackage{amsmath,amssymb,amsthm,graphicx,fixmath,tikz}
\usepackage{algorithm}
\usepackage{algorithmic}
\usepackage{verbatim}

\newcommand*{\Positives}{\mathbb{Z}^{+}}
\newcommand*{\Naturals}{\mathbb{N}}

\providecommand{\remove}[1]{}

\theoremstyle{plain}
\newtheorem{theorem}{Theorem}[section]
\newtheorem{lemma}[theorem]{Lemma}
\newtheorem{proposition}[theorem]{Proposition}

\newtheorem{corollary}[theorem]{Corollary}

\theoremstyle{definition}
\newtheorem{definition}[theorem]{Definition}
\theoremstyle{remark}
\newtheorem{remark}[theorem]{Remark}
\newcommand{\R}{\mathcal{R}}
\newcommand{\E}{\mathcal{E}}
\newcommand{\I}{\mathcal{I}}
\newcommand{\calJ}{\mathcal{J}}

\newcommand{\chicf}{{{\chi_{\text{cf}}}}}

\newcommand*{\floor}[1]{\lfloor#1\rfloor}%
\newcommand*{\ceil}[1]{\lceil#1\rceil}%

\newcommand*{\length}{\textup{len}}%

\newcommand{\cardin}[1]{\lvert {#1} \rvert}
\newcommand{\card}[1]{\lvert {#1} \rvert}

\begin{document}

\title{Conflict-free coloring with respect to a subset of intervals}

\author{%
Panagiotis Cheilaris\thanks{Department of Informatics,
Universit\`a della Svizzera italiana,
6900 Lugano, Switzerland.
\texttt{cheilarp@usi.ch}} \and
Shakhar
Smorodinsky\thanks{Mathematics department, Ben-Gurion University,
Be'er Sheva 84105, Israel. \texttt{shakhar@math.bgu.ac.il}}}

\date{}
\maketitle

\begin{abstract}
Given a hypergraph $H=(V,\cal E)$, a coloring of its vertices is
said to be conflict-free if for every hyperedge $S \in \cal E$
there is at least one vertex in $S$ whose color is distinct from
the colors of all other vertices in $S$.
The discrete interval hypergraph $H_n$ is the hypergraph with vertex set
$\{1,\dots,n\}$ and hyperedge set the family of all subsets of consecutive
integers in $\{1,\dots,n\}$.
We provide a polynomial time algorithm for conflict-free coloring
any subhypergraph of $H_n$, we show that the algorithm
has approximation ratio 2, and we prove that our analysis is tight, i.e.,
there is a subhypergraph for which the algorithm computes a solution which uses
twice the number of colors of the optimal solution.
We also show that the problem of deciding whether a given subhypergraph of $H_n$ can
be colored with at most $k$ colors
has a quasipolynomial time algorithm.
\end{abstract}

%\thispagestyle{empty}
%\setcounter{page}{0}
%\newpage

\section{Introduction}%
\label{sec:intro}

A hypergraph $H$ is a pair $(V,\E)$, where $V$ is a finite set
and $\E$ is a family of non-empty subsets of $V$. We denote by
$\Positives$ the set of positive integers and by $\Naturals$
the set of non-negative integers.

\begin{definition}\label{def:origcf}
Let $H=(V,\E)$ be a hypergraph and let $C$ be a coloring
$C \colon V \rightarrow \Positives$:
We say that $C$ is a {\em conflict-free coloring} (cf-coloring in short) if for
every hyperedge $e \in \E$ there exists a color $i \in \Positives$
such that $\card{e \cap C^{-1}(i) }=1$. That is, every
hyperedge $e \in \E$ contains some vertex whose color is unique
in $e$.
%We denote by
%$\chicf(H)$ the minimum integer $k$ for which $H$ admits a
%cf-coloring with a total of $k$ colors.
\end{definition}

%We also define another related coloring:

%\begin{definition}
%Let $H=(V,\E)$ be a hypergraph and let $C$ be a coloring
%$C \colon V \rightarrow \Positives$:
%We say that $C$ is a {\em unique-maximum coloring} (um-coloring in short) if for
%every hyperedge $S \in \E$ we have $\card{S \cap C^{-1}(\max\{C(v) \mid v \in S\}) }=1$.
%That is, the maximum color occurring in every
%hyperedge $S \in \E$ occurs exactly one time in $S$.
%We denote by
%$\chium(H)$ the minimum integer $k$ for which $H$ admits a
%um-coloring with a total of $k$ colors.
%\end{definition}

The study of cf-coloring was initiated in the work of Even et al.
\cite{ELRS} and of Smorodinsky \cite{SmPHD} and was extended in
numerous other works (c.f.,
\cite{cf9,AS08,BCOScpc2010,CheilarisCUNYthesis2009,%
cf7,CKS2009talg,cf5,HS02,Lev-TovP09,CFPT09,cf1}).
The study was initially motivated by its application to frequency
assignment for cellular networks. A cellular network consists of
two kinds of nodes: \emph{base stations} and \emph{mobile
clients}. Base stations have fixed positions, modeled by a finite set of
points in the plane, and provide the
backbone of the network. Every
base station emits at a fixed frequency. If a client
wants to establish a link with a base station it has to tune
itself to this base station's frequency. Clients, however, can be
in the range of many different base stations. To avoid
interference, the system must assign frequencies to base stations
in the following way: For any closed disk $d$ in the plane
(representing the communication range of a client located at the
center of this disk), there must be at least one
base station which is contained in $d$ and has a frequency that is not
used by any other base station contained in $d$. Since frequencies
are limited and costly, a scheme that reuses frequencies, where
possible, is desirable.

Here is a more general, formal definition:
Let $P$ be a set of $n$
points in the plane and let $\R$ be a family of regions in the
plane (e.g., all closed discs). We
denote by $H=H_{\R}(P)$ the hypergraph on the set $P$ whose
hyperedges are all subsets $P'$ that can be cut off from $P$ by a
region in $\R$. That is, all subsets $P'$ such that there exists
some region $r \in \R$ with $r\cap P = P'$. We refer to such a
hypergraph as the hypergraph \emph{induced by $P$ with respect to
$\R$}.

Now, consider the hypergraph induced by a set of $n$ \emph{collinear}
points with respect to the family of closed disks in the plane.
It is not difficult to see that this hypergraph is isomorphic to the
hypergraph induced by a set of $n$ real numbers with respect to the family of
closed intervals,
which is also isomorphic to the following discrete interval hypergraph.

\begin{definition}
Let $[n] = \{1,\dots,n\}$.
For $s \leq t$, $s, t \in [n]$, we define the {(discrete) interval}
$[s,t] = \{i \in [n] \mid s \leq i \leq t\}$.
The \emph{discrete interval hypergraph} $H_n$ has vertex set $[n]$ and
hyperedge set
$\I_n = \{  [s,t]  \mid {s \leq t}\text{, }{s, t \in [n]} \}$.

\end{definition}

It is not difficult to prove that
$\lfloor \log_2 n \rfloor + 1$ colors are necessary and sufficient
in order to cf-color $H_n$
(see, e.g., \cite{ELRS}).
An online variation of this cf-coloring problem in which
vertices appear one by one and the algorithm has to commit to
a color for each point as soon as it appears,
maintaining the conflict-free property
of the point set at every time,
was introduced in \cite{cf7} and further studied in
\cite{BCS2008talg}.

In this paper, we are interested in cf-coloring subhypergraphs of
$H_n$ of the following form:
$H = ([n], I)$, where $I \subseteq \I_{n}$.
Then, $H$ is a hypergraph induced by $n$ points on the real line
with respect to a \emph{subset} of all possible intervals.
Cf-colorings of such hypergraphs were studied
in the online setting in \cite{BCS2008talg}.
Katz et al., in \cite{cf2}, claim a 4-approximation polynomial time
cf-coloring for any such hypergraph $H$ (in the offline setting).
Studying cf-coloring for subhypergraphs of geometric hypergraphs
can be justified by applications where only a given subset of the
hyperedge set is required to have the conflict-free property.

In section~\ref{sec:hittingsetalg}, we describe an algorithm
for computing cf-colorings for general hypergraphs,
based on hitting sets.
In section~\ref{sec:2approx}, we show how the above algorithm
and an appropriate choice of the hitting set can give a 2-approximation
polynomial time
algorithm for cf-coloring a subhypergraph of the discrete interval
hypergraph, improving on the 4-approximation algorithm of Katz et al.
In section~\ref{sec:tight}, we show that the above analysis is tight, i.e.,
there are subhypergraphs of $H_n$ for which
the algorithm computes a cf-coloring with twice the optimal
(minimum) number of colors.
In section~\ref{sec:quasip}, we show that the
decision problem whether a given
subhypergraph of $H_n$ can be cf-colored with
at most $k$ colors has a quasipolynomial time algorithm;
this implies that this decision problem is probably not
NP-complete.

\section{A hitting-set algorithm for conflict-free coloring}%
\label{sec:hittingsetalg}

In this section, we present an algorithm for conflict-free
coloring a hypergraph. It is based on repeatedly computing a
minimal hitting set in hypergraphs.

\begin{definition}
A \emph{hitting set} of a hypergraph $H = (V,\E)$
is a subset $S \subseteq V$ such that for every $e \in \E$
there exists some $v \in S$ with $v \in e$. A hitting set $S$ is
\emph{minimal} if for every $v \in S$, $S \setminus \{v\}$
is not a hitting set.
\end{definition}

In the literature, a conflict-free coloring is an
assignment of colors (positive integers) to the vertices of the hypergraph.
In this work, we introduce and
consider a slight variation of conflict-free coloring, in which
we allow some vertices to not be assigned colors, as long as
in every hyperedge, there
exists a vertex with assigned color that is
uniquely occurring in the hyperedge.
In other words, we allow the coloring function
$C\colon V \to \Positives$ in definition~\ref{def:origcf}
to be a partial function. Alternatively, we
can use a special color `0' given to vertices that are not
assigned any positive color and obtain a total function $C \colon V
\to \Naturals$. Then, we arrive at the following variant of
definition~\ref{def:origcf}.

\begin{definition}
Let $H=(V,\E)$ be a hypergraph and let
$C \colon V \rightarrow \Naturals$:
We say that $C$ is a {\em conflict-free coloring} if for
every hyperedge $S \in \E$ there exists a color $i \in \Positives$
such that $\card{S \cap C^{-1}(i) }=1$.
We denote by
$\chicf(H)$ the minimum integer $k$ for which $H$ admits a
cf-coloring with colors in $\{0,\dots,k\}$.
\end{definition}

\begin{remark}
We claim that this variation of conflict-free coloring,
with the partial coloring function or the placeholder color `0',
is interesting from the point of view
of applications. As mentioned in section~\ref{sec:intro},
vertices model base stations in a cellular network. A vertex with
no positive color assigned to it can model a situation where a base station
is not activated at all, and therefore the base station does not
consume energy.
One can also think of a bi-criteria optimization problem where a
conflict-free assignment of frequencies has to be found with small
number of frequencies (in order to conserve the frequency spectrum)
and few activated base stations (in order to conserve energy).
\end{remark}

We describe algorithm~\ref{alg:hitset} for conflict-free coloring any
hypergraph $H = (V,\E)$.

\begin{algorithm}[htbp]
\caption{A hitting set algorithm for conflict-free coloring $H=(V,\E)$}%
\label{alg:hitset}
\begin{algorithmic}
\STATE $\ell \gets 0$; $V^0 \gets V$; $\E^0 \gets \E$
\WHILE{$\E^\ell \neq \emptyset$}
  \STATE $S^\ell \gets \text{a minimal hitting set for $(V^\ell, \E^\ell)$}$
  \STATE color every $v \in V^\ell \setminus S^\ell$ with color $\ell$
  \STATE $V^{\ell+1} \gets S^\ell$
  \STATE $\E^{\ell+1} \gets \{e \cap S^{\ell} \mid e \in \E^\ell \text{ and }
              |e \cap S^{\ell}| > 1 \}$
  \STATE $\ell \gets \ell+1$
\ENDWHILE
\STATE % one-line if
\textbf{if} $V^\ell \neq \emptyset$
\textbf{then} color every $v \in V^\ell$ with color $\ell$
\textbf{end if}
\end{algorithmic}
\end{algorithm}

\begin{lemma}
Algorithm~\ref{alg:hitset} terminates.
\end{lemma}
\begin{proof}
At every iteration of the loop, there is some hyperedge $e \in \E^\ell$
for which $|e \cap S^{\ell}| = 1$. 
This follows from the minimality of $S^\ell$.
Thus, $|\E^\ell| > |\E^{\ell+1}|$. Therefore, the number of
hyperedges decreases at every iteration of the loop, and
necessarily reaches zero after a finite number of iterations of
the loop.
\end{proof}

\begin{lemma}
Algorithm~\ref{alg:hitset} produces a conflict-free coloring.
\end{lemma}
\begin{proof}
We first show that
for every hyperedge $e \in \E$, there is some $\ell$ for which
$\cardin{{e} \cap S^{\ell}} = 1$.
Notice that for every iteration $i>0$, we have
$S^{i-1} \supseteq S^i$.
If $\cardin{e \cap S^0} > 1$, consider the maximum $i$ for which
$\cardin{e \cap S^i} > 1$.
Then, hyperedge $e \cap S^i = e \cap V^{i+1}$
belongs to $\E^{i+1}$ and has to be hit by $S^{i+1}$, i.e.,
$(e \cap S^{i}) \cap S^{i+1} = e \cap S^{i+1}$ is non-empty
and thus $\cardin{e \cap S^{i+1}} = 1$,
because of the maximality of $i$.

Let $v$ be the one element of ${e} \cap S^{\ell}$.
Vertex $v$ is colored with some color greater than
$\ell$ by the algorithm and all other vertices of $e$ are colored
with colors which are at most of value $\ell$. Thus, $e$ has the
conflict-free property.
\end{proof}

\section{A 2-approximation algorithm for a set of intervals}%
\label{sec:2approx}

We use algorithm~\ref{alg:hitset}, described in the previous section,
to conflict-free color a subhypergraph of $H_n$ which is comprised of a given
subset $I \subseteq \I_n$ of intervals.
It is necessary to specify how to compute the minimal hitting set.

The minimal hitting set $S$ is computed as follows 
(in fact, we compute a minimum cardinality hitting set, but we do
not need this stronger fact):
\begin{quote}
First, we compute a special independent set of intervals $F \subseteq I$
(i.e., in $F$ no two intervals have a common vertex).
We compute this independent set $F$ of intervals incrementally.
Initially, there is nothing in the independent set. We scan
vertices from $1$ to $n$ and we include in the independent set the
interval $[i,j] \in I$ with minimum $j$ such that $[i,j]$ does not
intersect anything already in the independent set. After computing
$F$, for every interval $[i,j] \in F$, we take in $S$ the vertex $j$
(i.e., the maximum or rightmost vertex).
\end{quote}

\begin{lemma}
$S$ is a minimal hitting set.
\end{lemma}
\begin{proof}
Set $S$ is a hitting set because no interval is completely contained
between two vertices in $S$, no interval ends before the first
interval in $F$, and no interval starts after the last interval in
$F$; otherwise such intervals would be chosen in the independent set
$F$. Set $S$ is minimal, because removing any element $j$ of it, means
that the interval with right endpoint $j$ in $F$ is not hit any
more.
\end{proof}

\begin{remark}
The computation of the maximal (in fact maximum)
independent set of intervals given above
is also known as a solution to the activity selection problem.
See for example \cite[section~16.1]{CLRS}.
\end{remark}

Notice that the time complexity of the algorithm is $O(n\log{n})$:
We sort the intervals according to their right endpoints.
Then, at every iteration of the loop we can choose the hitting set
in linear time. There is at most a logarithmic number of
iterations of the loop, because
$\chicf(H) \leq \chicf(H_n) = \floor{\log_{2}{n}}+1$.

We intend to compare colorings produced by the above algorithm
with optimal colorings.
We define recursively the following families of sets of intervals
of $\Positives$.

\begin{definition}
Family $\calJ_1$ exactly contains all singleton sets of intervals. 
For $k>1$, 
%family $\calJ_k$ exactly contains sets of intervals 
%with the following property:
set of intervals $I$ is in family $\calJ_k$ 
if and only if it can be expressed as a
union $I = L \cup R \cup \{\iota\}$, 
where both $L$, $R \in \calJ_{k-1}$, 
no interval from $L$ has a
common point with an interval from $R$,
and interval $\iota$ 
includes every interval 
in $L$ and every interval in $R$.

We refer to a set of intervals in family $\calJ_k$ as a 
\emph{$\calJ_k$ configuration}.
\end{definition}

\begin{lemma}
Any conflict-free coloring uses at least $k$ colors for a set of
intervals that is a superset of a $\calJ_k$ configuration.
\end{lemma}
\begin{proof}
We use induction on $k$. For $k=1$, the statement is trivially true.
Assume it is true for $k$, we will prove it for $k+1$. Assume, for
the sake of contradiction, that there is a conflict-free coloring $C$
with just $k$ colors of a set of intervals $I'$ that is a superset
of a $\calJ_{k+1}$ configuration $I$.
Then, by definition of $\calJ_{k+1}$,  
$I = L \cup R \cup \{\iota\}$,
where both $L$, $R \in \calJ_{k}$, 
no interval from $L$ has a
common point with an interval from $R$,
and interval $\iota$ 
includes every interval 
in $L$ and every interval in $R$.
By the inductive hypothesis, 
the points contained in intervals of $L$ use $k$ colors 
and also 
the points contained in intervals of $R$ use $k$ colors.
The above two pointsets are disjoint
and the interval $\iota$ includes both pointsets.
As a result, $\iota$ is not
conflict-free colored, which is a contradiction.
\end{proof}

We are now ready to bound the approximation ratio of the proposed
algorithm.

\begin{theorem}
The conflict-free coloring algorithm for hypergraphs with respect to
a subset of intervals is a 2-approximation algorithm.
\end{theorem}
\begin{proof}
It is enough to prove that if
some hyperedge (or interval), say $\iota$,
reaches iteration with $\ell=k-1$ of the loop (i.e.,
the algorithm uses at least $k$ colors),
then the input contains as a subset a
$\calJ_{\lceil k/2 \rceil}$ configuration 
and moreover this configuration is entirely contained in $\iota$.

We prove it by induction. For $k=1,2$, it is true,
because there is at least one interval in the input, and therefore
at least one non-zero color is needed in any optimal coloring.
For $k > 2$, assume there is a vertex $v$ that gets color $k$. Then
at iteration with $\ell=k-1$ of the loop there is an interval $\iota$ with
its rightmost vertex being $v \in S^{\ell}$ (see figure~\ref{fig:jkfind}).

\begin{figure}[htbp]
\centering
\begin{tikzpicture}
% vertices
\coordinate (u) at ( 2,0);
\coordinate (w) at ( 5,0);
\coordinate (v) at (10,0);
\node [below] at (u) {$u$};
\node [below] at (w) {$w$};
\node [below] at (v) {$v$};
% vertical lines from each vertex
\begin{scope}[gray,dashed]
  \draw ( 2,0)--( 2,3.3);
  \draw ( 5,0)--( 5,3.3);
  \draw (10,0)--(10,3.3);
\end{scope}
% vertices
\fill (u) circle (2pt);
\fill (w) circle (2pt);
\fill (v) circle (2pt);
% intervals
\draw [thick] (1,1)--(10,1);
\node [right] at (10,1) {$\iota$};
\draw [thick] (0,2)--(2,2);
\node [right] at (2,2) {$\iota'_1$};
\draw [thick] (4,2)--(10,2);
\node [right] at (10,2) {$\iota'_2$};
\draw [thick] (3,3)--(5,3);
\node [right] at (5,3) {$\iota''_1$};
\draw [thick] (8,3)--(10,3);
\node [right] at (10,3) {$\iota''_2$};
\end{tikzpicture}
\caption{Intervals in an input using $k$ colors}%
\label{fig:jkfind}%
\end{figure}
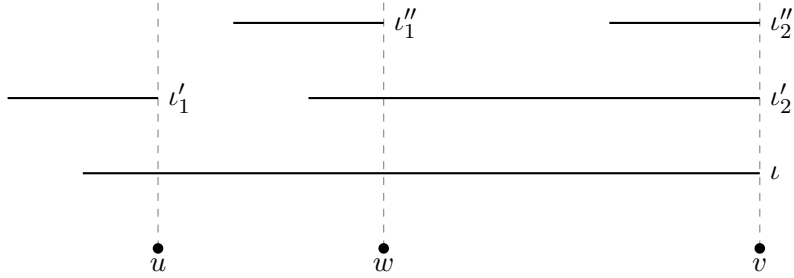

Since $\iota$ was not removed
in the previous iteration $\ell-1$, there were two vertices of
$\iota$ in
$S^{\ell-1}$, say $u$ and $v$, with $u < v$. Also, since $u$ and $v$
are in $S^{\ell-1}$ there are two intervals with them as right
endpoints in the independent set computed at iteration $\ell-1$, say
$\iota'_1$ and $\iota'_2$. Since $\iota'_2$ was not removed in the iteration
$\ell-2$, there were two vertices of $\iota'_2$ in $S^{\ell-2}$,
say $w$ and $v$,
with $u < w < v$. Also, since $u$, $w$, and $v$ are in  $S^{\ell-2}$
there are three intervals with them as right endpoints in the
independent set computed at iteration $\ell-2$; call $\iota''_1$ the one
ending at $w$ and $\iota''_2$ the one ending at $v$. Since the three
intervals are independent,  $\iota''_1$ and $\iota''_2$ start after $u$,
therefore they are fully contained in $\iota$ (which contains $u$).
By the inductive hypothesis, since each of $\iota''_1$, $\iota''_2$ reach
iteration $\ell-2$, each of them entirely contains a
$\calJ_{\lceil (k-2)/2 \rceil}$ configuration, and, since $\iota''_1$
and $\iota''_2$ are disjoint, together with $\iota$ they constitute a
$\calJ_{\lceil k/2 \rceil}$ configuration.
\end{proof}

\section{A tight instance for the 2-approximation algorithm}%
\label{sec:tight}

For $k \geq 2$, we intend to define an input $I_k$ that is a tight instance
for the approximation algorithm, i.e., an instance that forces the algorithm
to use at least twice the number of colors in an optimal coloring.
Before doing that, we define some notation
that will prove useful.

\begin{definition}
Given a set of intervals $I$ and a natural number $d$, we define
$I^{+d}$ to be the set of intervals, where all intervals of $I$
are shifted $d$ to the right, i.e.,
\[I^{+d} = \{[i+d,j+d] \mid [i,j] \in I \}.\]
%for every $[i,j] \in I$, there
%is $[i+d, j+d] \in I^{+d}$, and there are no other intervals in
%$I^{+d}$.
\end{definition}

\begin{definition}
Given a set of intervals $I$, we define the \emph{length} of $I$, denoted
$\length(I)$ to be the rightmost point occurring in any of the
intervals of $I$ minus the leftmost point occurring in any of the
intervals of $I$ plus one.
\end{definition}

Now, we are ready to proceed with the definition of the tight
instance.

\begin{definition}
For $k=2$ the input $I_2$ has length equal to
four and consists of three intervals.
\[I_2 = \{[1,2], [3,3], [2,4] \}\]
For $k > 2$ the input is defined recursively as follows.
\[I_{k+1} = I_k \cup I_k^{+\length(I_k)} \cup
  \{[\length(I_k)-k+1, 2\length(I_k)+1]\}\]
\end{definition}

Abusing notation,
we call the $I_k$ component the \emph{left} $I_k$ part of $I_{k+1}$
and the $I_k^{+\length(I_k)}$ component the \emph{right} $I_k$ part of
$I_{k+1}$. These left and right parts are disjoint.
%Inputs $I_2$, $I_3$, $I_4$, $I_5$ are shown in figures
%\ref{fig:inputk2}, \ref{fig:inputk3}, \ref{fig:inputk4},
%\ref{fig:inputk5}, respectively.
Input $I_4$ is shown in figure~\ref{fig:inputk4}.
Moreover, in the figure,
under the vertices of the input we give the coloring produced by 
the 2-approximation
algorithm and then an optimal conflict-free coloring.

\begin{comment}
\begin{figure}[htbp]
\centering
\input{inputk2ext}

%  3*1.0 + 0.3 = 3.3

\medskip

\makebox[3.3cm][l]{0\hfill 1\hfill 2\hfill 0 }%

\medskip

\makebox[3.3cm][l]{1\hfill 0\hfill 1\hfill 0 }%

\caption{Input $I_2$ and conflict-free colorings}%
\label{fig:inputk2}%
\end{figure}

\begin{figure}[htbp]
\centering
\input{inputk3ext}

%  8*1.0 + 0.3 = 8.3
\makebox[8.3cm][l]{%
0\hfill 1\hfill 2\hfill 0\hfill
0\hfill 1\hfill 3\hfill 0\hfill 0\ }

\makebox[8.3cm][l]{%
1\hfill 0\hfill 1\hfill 2\hfill 1\hfill 0\hfill 1\hfill 0\hfill
0\ }

\caption{Input $I_3$ and conflict-free colorings}%
\label{fig:inputk3}%
\end{figure}
\end{comment}

\begin{figure}[htbp]
\centering
\begin{tikzpicture}[scale=0.7]
\foreach \x in { 1, ..., 19 }
    \fill[black] (\x,0) circle (0.05) ;
\draw ( 0.9 , 1 ) -- ( 2.1 , 1 ) ;
\draw ( 2.9 , 1 ) -- ( 3.1 , 1 ) ;
  \draw ( 1.9 , 2 ) -- ( 4.1 , 2 ) ;
\draw ( 4.9 , 1 ) -- ( 6.1 , 1 ) ;
\draw ( 6.9 , 1 ) -- ( 7.1 , 1 ) ;
  \draw ( 5.9 , 2 ) -- ( 8.1 , 2 ) ;
    \draw ( 2.9 , 3 ) -- ( 9.1 , 3 ) ;
\draw ( 9.9 , 1 ) -- ( 11.1 , 1 ) ;
\draw ( 11.9 , 1 ) -- ( 12.1 , 1 ) ;
  \draw ( 10.9 , 2 ) -- ( 13.1 , 2 ) ;
\draw ( 13.9 , 1 ) -- ( 15.1 , 1 ) ;
\draw ( 15.9 , 1 ) -- ( 16.1 , 1 ) ;
  \draw ( 14.9 , 2 ) -- ( 17.1 , 2 ) ;
    \draw ( 11.9 , 3 ) -- ( 18.1 , 3 ) ;
      \draw ( 6.9 , 4 ) -- ( 19.1 , 4 ) ;
\end{tikzpicture}

%  18*0.7 + 0.3 = 12.9
\makebox[12.9cm][l]{%
0\hfill 1\hfill 2\hfill 0\hfill 0\hfill 1\hfill 3\hfill 0\hfill
0\hfill 0\hfill 1\hfill 2\hfill 0\hfill 0\hfill 1\hfill 4\hfill
0\hfill 0\hfill 0\ }

\makebox[12.9cm][l]{%
1\hfill 0\hfill 1\hfill 2\hfill 1\hfill 0\hfill 1\hfill 0\hfill
0\hfill 1\hfill 0\hfill 1\hfill 2\hfill 1\hfill 0\hfill 1\hfill
0\hfill 0\hfill 0\ }

\caption{Input $I_4$, algorithm cf-coloring, and optimal cf-coloring}%
\label{fig:inputk4}%
\end{figure}
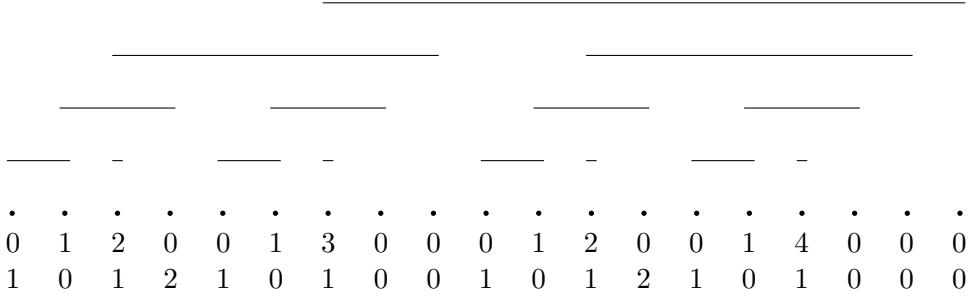

\begin{comment}
\begin{figure}[htbp]
\centering
\input{inputk5ext}
\medskip

%0 1 2 0 0 1 3 0 0 0 1 2 0 0 1 4 0 0 0
%0 1 2 0 0 1 3 0 0 0 1 2 0 0 1 5 0 0 0 0 {\ }

%  38*0.344 + 0.4 = 13.472
\makebox[13.472cm][l]{%
0\hfill 1\hfill 2\hfill 0\hfill 0\hfill 1\hfill 3\hfill
0\hfill 0\hfill 0\hfill 1\hfill 2\hfill 0\hfill 0\hfill
1\hfill 4\hfill 0\hfill 0\hfill 0\hfill
0\hfill 1\hfill 2\hfill 0\hfill 0\hfill 1\hfill 3\hfill
0\hfill 0\hfill 0\hfill 1\hfill 2\hfill 0\hfill 0\hfill
1\hfill 5\hfill 0\hfill 0\hfill 0\hfill 0%
%}
\ \ }

\medskip

\makebox[13.472cm][l]{%
1\hfill 0\hfill 1\hfill 2\hfill 1\hfill 0\hfill 1\hfill 0\hfill 0\hfill
1\hfill 0\hfill 1\hfill 2\hfill 1\hfill 0\hfill 1\hfill 0\hfill 0\hfill 3\hfill
1\hfill 0\hfill 1\hfill 2\hfill 1\hfill 0\hfill 1\hfill 0\hfill 0\hfill
1\hfill 0\hfill 1\hfill 2\hfill 1\hfill 0\hfill 1\hfill 0\hfill 0\hfill 0\hfill
0%
%}
\ \ }

\caption{Input $I_5$ and conflict-free colorings}%
\label{fig:inputk5}%
\end{figure}
\end{comment}

It is not difficult to see that the length of the instance
satisfies the recurrence relation
\begin{equation}
\length(I_{k+1}) = 2\length(I_k) + 1 , \label{eq:Iklength}
\end{equation}
which implies, since $\length(I_2) = 4$, that
\(\length(I_k) = 5 \cdot 2^{k-2} - 1\).
%\[\length(I_k) = 5 \cdot 2^{k-2} - 1.\]

Another notion that will prove useful is the \emph{level} of each
interval in the above instance that we define in the following.

\begin{definition}
In input $I_2$, intervals $[1,2]$ and $[3,3]$ are of level 1 and
interval $[2,4]$ is of level 2. In the recursively defined instance
\[I_{k+1} = I_k \cup I_k^{+\length(I_k)} \cup
  \{[\length(I_k)-k+1, 2\length(I_k)+1]\}\]
the intervals of the $I_k$ part have
the same levels as the corresponding intervals in the $I_k$ instance,
the intervals of the $I_k^{+\length(I_k)}$ part have the same levels
as the corresponding intervals of the $I_k$ instance before the
`$+\length(I_k)$' operation,
and interval $[\length(I_k)-k+1, 2\length(I_k)+1]$ has level $k+1$.
\end{definition}

In fact, in %
figure~\ref{fig:inputk4}
%figures~\ref{fig:inputk2}, \ref{fig:inputk3},
%\ref{fig:inputk4}, and~\ref{fig:inputk5}
%
the vertical coordinate
of each interval signifies its level, with higher intervals having
higher level.

\begin{lemma}\label{lem:leftmostklevel}%
For $k \geq 3$, in $I_k$, the leftmost point of the level $k$
interval is the same as the rightmost level $1$ interval in the
left $I_{k-1}$ part of $I_k$.
\end{lemma}
\begin{proof}
We prove by induction that the rightmost level $1$
interval of the left $I_{k-1}$ part of $I_k$ is at position
$\length(I_{k-1})-(k-1)+1$.
For $I_3$, the rightmost level 1 interval of the left $I_2$ part
of $I_3$ consists of point $4-(3-1)+1=3$.
By the inductive hypothesis, the rightmost level $1$
interval of the left $I_{k-1}$ part of $I_k$ is at $\length(I_{k-1})-(k-1)+1$.
Then for $I_{k+1}$, the rightmost level $1$ interval of its
left $I_{k}$ part is at
\[
  \length(I_{k-1})-(k-1)+1 + \length(I_{k-1}) =
  (2\length(I_{k-1})+1) + k - 1 = \length(I_k)+k-1.
\]
The last equality is implied by equation~\eqref{eq:Iklength}.
\end{proof}

\begin{lemma}\label{lem:IkcontainsJk2}
Instance $I_k$ contains a $\calJ_{\ceil{k/2}}$ configuration
as a subset.
\end{lemma}
\begin{proof}
By induction. It is true for $k=2$ and $k=3$, because $I_2$
contains a $\calJ_1$ configuration and $I_3$ contains a $\calJ_2$
configuration. For $k>3$, in instance $I_k$, the interval of level
$k$ contains completely a copy of $I_{k-1}$, in which two disjoint
copies of $I_{k-2}$ are contained. By the inductive hypothesis, in
each copy of $I_{k-2}$, a $\calJ_{\ceil{(k-2)/2}}$ configuration
is contained. These two disjoint $\calJ_{\ceil{(k-2)/2}}$
configurations, together with the level $k$ interval constitute a
$\calJ_{\ceil{k/2}}$ configuration in $I_k$.
\end{proof}

\begin{lemma}\label{lem:Ikcolk2}
There is a conflict-free coloring of $I_k$ with $\ceil{k/2}$ colors.
\end{lemma}
\begin{proof}
We define recursively a coloring of $I_k$ that uses
$\ceil{k/2}$ colors and we prove by induction that it is
conflict-free.

For $k=2$ the coloring is $1010$, which can be easily checked to
be conflict-free.

If $k$ is odd, take a coloring of $I_{k-1}$ and in its rightmost
position use color $\ceil{k/2}$, concatenate a coloring of
$I_{k-1}$, and then concatenate color `0'. By induction, the left
$I_{k-1}$ part is conflict-free because we started with a conflict-free
coloring and we introduced a new color $\ceil{k/2}$, the right
$I_{k-1}$ part is conflict-free because it is colored with a
conflict-free coloring. The level $k$ interval
is conflict-free because of color $\ceil{k/2}$ that occurs
uniquely.

If $k$ is even, with $k>2$,
take a coloring of $I_{k-1}$, concatenate a coloring of
$I_{k-1}$, and then concatenate color `0'. By induction, the left
$I_{k-1}$ part is conflict-free because it is colored with a
conflict-free coloring, the right
$I_{k-1}$ part is conflict-free because it is colored with a
conflict-free coloring. The level $k$ interval
is conflict-free because of color $\ceil{k/2}$ that occurs in the
right $I_{k-1}$ part and because its leftmost point, by
lemma~\ref{lem:leftmostklevel}, is to the right of the
$\ceil{k/2}$ color occurring in the left $I_{k-2}$ part of the
left $I_{k-1}$ part.
\end{proof}

\begin{corollary}
An optimal coloring of $I_k$ uses $\ceil{k/2}$ colors.
\end{corollary}

We now describe a family of hypergraphs that arise after the first
iteration of the while loop of the 2-approximation algorithm,
if the initial input is $I_k$.

\begin{definition}
The instance $L_0$ is on one vertex, namely the vertex set is \{1\},
and contains no
interval, i.e, $L_0 = \{\}$. The length of instance $L_0$ is
defined to be 1. For $k>0$, $L_{k+1}$ is defined
recursively, as follows.
\[L_{k+1} = L_k \cup L_k^{+\length(L_k)} \cup
            \{[\length(L_k),2\length(L_k)]\}
\]
\end{definition}
It is not difficult to see that the length satisfies the
recurrence relation $\length(L_{k+1})=2\length(L_k)$,
which implies $\length(L_k)=2^k$. We say that
$L_{k+1}$ consists of a left $L_k$ part, a right $L_k$ part,
and the interval $[2^{k},2^{k+1}]$.

\begin{proposition}\label{prop:Ikalgk}
The 2-approximation algorithm colors $I_k$ with $k$ colors.
\end{proposition}
\begin{proof}
Assume input $I_k$ is given to the 2-approximation algorithm.
In the iteration of the while loop
where the algorithm colors points with color
$\ell$ ($\ell = 0, 1,
\dots$), the algorithm considers a
hypergraph $H_\ell$. We will prove that the algorithm
considers the hypergraphs
\[H_0 = I_k, H_1 = L_{k-1}, \dots, H_{k-1} = L_1, H_k = L_0, \]
and then it terminates, i.e., it uses $k$ colors. We say that
$H_i$ is \emph{followed} by $H_{i+1}$, to show that two
hypergraphs $H_i$, $H_{i+1}$ are considered successively by the
algorithm, in that order.

First, we prove that for every $k \geq 2$, $I_k$ is followed by
$L_{k-1}$, by induction on $k$.
It is not difficult to see that, when $I_k$ is considered,
the independent set of intervals chosen consists of all level 1
intervals of $I_k$ and the hitting set that is chosen
consists of the right endpoints of all level 1 intervals of $I_k$
(a formal proof can be carried out by induction on $k$).
For $k=2$ it is not difficult to check that $I_2$ is followed by
$L_1$. For $k>2$, $I_k$ consists of a left $I_{k-1}$ part which
induces a left $L_{k-2}$ part and a right $I_k$ part, which
induces a right $L_{k-2}$ part (we use the inductive hypothesis).
From lemma~\ref{lem:leftmostklevel},
the leftmost point of the level $k$
interval is the same as the rightmost level $1$ interval in the
left $I_{k-1}$ part of $I_k$, and therefore the level $k$
interval induces
an interval that starts from the last point of the left $L_{k-2}$
part of the hypergraph that follows $I_k$
and ends at the last point of the right $L_{k-2}$
part of the hypergraph that follows $I_k$.
To summarize, the $I_k$ is followed by a left $L_{k-2}$
part, a right $L_{k-2}$ part and interval $[2^{k-2},2^{k-1}]$, i.e., it is
$L_{k-1}$.

Then, we prove that for $k > 0$, $L_k$ is followed by $L_{k-1}$,
by induction on $k$.
For $k=1$, it is not difficult to see that for $L_1$ the interval
$[1,2]$ is chosen and its right endpoint, i.e., $2$,
makes up the hitting set. Then, easily, $L_1$ is followed by $L_0$.
For $k>1$,
when $L_k$ is considered, the independent set of intervals that is
chosen consists of the intervals of length two
of the left $L_{k-1}$ part
\[\{[1,2],[3,4], \dots, [2^{k-1}-1,2^{k-1}]\}\]
and the intervals of length two of the right
$L_{k-1}$ part
\[\{[2^{k-1}+1,2^{k-1}+2],[2^{k-1}+3,2^{k-1}+4],
     \dots, [2^{k}-1,2^{k}]\} . \]
Therefore the hitting set is
\[ \{2,4, \dots 2^{k-1}\} \cup
   \{2^{k-1}+2, 2^{k-1}+4, \dots, 2^k\} =
   \{i \colon \text{odd} \mid 2 \leq i \leq 2^k\}
\]
and consists of $2^{k-1}$ elements.
By induction, after removal of the points of the hitting set,
the left $L_{k-1}$ part induces a $L_{k-2}$ part, and the right
$L_{k-1}$ part induces a $L_{k-2}$ part. The interval
$[2^{k-1},2^k]$ of $L_k$ contains all points in
$\{2^{k-1}+2,2^{k-1}+4,\dots,2^k\}$
of the right $L_{k-1}$ part
and just point $2^{k-1}$ of the left $L_{k-1}$ part, and therefore
induces $[2^{k-2}, 2^{k-1}]$ in the hypergraph that follows $L_k$.
To summarize, the $L_k$ is followed by a left $L_{k-2}$
part, a right $L_{k-2}$ part and interval $[2^{k-2},2^{k-1}]$, i.e., it is
$L_{k-1}$.

Finally, we prove that when $L_0$ is reached, no hypergraph
follows, and the algorithm
terminates. This is true, because $L_0$ contains no interval
(hyperedge).
\end{proof}

\begin{remark}
From the above proof of proposition~\ref{prop:Ikalgk},
it is immediate that if $L_k$ is given as an input to the
2-approximation algorithm, the following sequence of hypergraphs
\[H_0 = L_k, H_1 = L_{k-1}, \dots, H_{k-1} = L_1, H_k = L_0 \]
is considered in the iterations of the while loop.
Moreover, it can also be proved, with a proof similar to those
of lemmata~\ref{lem:IkcontainsJk2} and~\ref{lem:Ikcolk2},
that an optimal coloring for $L_k$
uses $\ceil{k/2}$ colors.
Therefore, the family of instances $L_k$ is also a family of tight
instances for the 2-approximation algorithm.
However, the family of instances $I_k$ has the additional property
that no two intervals in it share a common right endpoint.
\end{remark}

\section{A quasipolynomial time algorithm}%
\label{sec:quasip}

Consider the decision problem \textsc{CFSubsetIntervals}:
\begin{quote}
``Given a subhypergraph $H=([n],I)$
of the discrete interval hypergraph $H_n$ and a natural number $k$,
is it true that
$\chicf(H) \leq k$?''
\end{quote}
Notice that the above problem is non-trivial
only when $k < \floor{\log_2 n}+1$; if
$k \geq \floor{\log_2 n} + 1$ the answer is always yes, since
$\chicf(H_n) = \floor{\log_2 n} + 1$.

Algorithm~\ref{alg:nondet} is a non-deterministic
algorithm for \textsc{CFSubsetIntervals}.

The algorithm scans points from $1$ to $n$, tries
non-deterministically every color in $\{0,\dots,k\}$
at the current point and checks if all intervals in $I$
ending at the current point have the conflict-free property.
If some interval in $I$ has not the conflict-free property under
a non-deterministic assignment, the algorithm answers `no'.
If all intervals in $I$ have the conflict-free property
under some non-deterministic assignment, the algorithm answers
`yes'.

We check if an interval in $I$
that ends at the current point, say $t$,
has the conflict-free property
in the following space-efficient way.
For every color $c$ in $\{0,\dots,k\}$, we keep track of:
\begin{enumerate}
\item[(a)]
the closest point to $t$ colored with $c$ in variable $p_c$
and
\item[(b)]
the second closest point to $t$ colored with $c$ in variable
$s_c$.
\end{enumerate}
Then, color $c$ is occurring exactly one time in $[j,t] \in I$
if and only if $s_c < j \leq p_c$.

\begin{algorithm}[htb!]
\caption{A non-deterministic algorithm deciding whether
         $\chicf(H) \leq k$ for $H=([n],I)$}
\label{alg:nondet}
\begin{algorithmic}
  \FOR{$c \gets 0 \text{ to } k$}
    \STATE $s_c \gets 0$%; $p_c \gets 0$
    \STATE $p_c \gets 0$
  \ENDFOR
  \FOR{$t \gets 1 \text{ to } n$}
     \STATE choose $c$ non-deterministically from $\{0,\dots,k\}$
     \STATE $s_{c} \gets p_{c}$%; $p_{c} \gets t$
     \STATE $p_{c} \gets t$
     \FOR{$j \in \{j \mid [j,t] \in I\}$}
        \STATE IntervalConflict $\gets$ True
	\FOR{$c \gets 1 \text{ to } k$}
	   \IF{$s_c < j \leq p_c$}
	      \STATE IntervalConflict $\gets$ False
	      %\STATE \textbf{break} out of for $c$ loop
	   \ENDIF
	\ENDFOR % c
	\IF {IntervalConflict}
	   \RETURN{NO}
	\ENDIF
     \ENDFOR % j
  \ENDFOR
  \RETURN{YES}
\end{algorithmic}
\end{algorithm}

\begin{lemma}
The space complexity of algorithm \ref{alg:nondet} is $O(\log^{2}{n})$.
\end{lemma}
\begin{proof}
Since $k = O(\log n)$ and each point position can be encoded with
$O(\log n)$ bits, the arrays $p$ and $s$ (indexed by color)
take space $O(\log^2 n)$. All other variables in the algorithm can
be implemented in $O(\log n)$ space.
Therefore the above non-deterministic algorithm has space
complexity $O(\log^2 n)$.
\end{proof}

\begin{corollary}
\textsc{CFSubsetIntervals} has a quasipolynomial time deterministic
algorithm.
\end{corollary}
\begin{proof}
By standard computational complexity theory arguments
(see, e.g., \cite{PapadimitriouCC}), we can transform
algorithm~\ref{alg:nondet}
to a deterministic algorithm solving the same problem with time
complexity $2^{O(\log^2 n)}$, i.e., \textsc{CFSubsetIntervals} has a
quasipolynomial time deterministic algorithm.
\end{proof}

\section{Discussion and open problems}

The exact complexity of computing an optimal cf-coloring for a
subhypergraph of the discrete interval hypergraph remains an
open problem. We have provided a 2-approximation algorithm.
One might try to improve the approximation ratio, find a
polynomial time approximation scheme, or even find a polynomial
time exact algorithm.
The last possibility is supported by the fact that
the decision version of the problem, \textsc{CFSubsetIntervals},
is unlikely to be NP-complete, unless NP-complete problems have
quasipolynomial time algorithms.

It would also be interesting to study the complexity
of computing optimal conflict-free colorings for subhypergraphs of
other geometric hypergraphs,
like the hypergraph induced by a set of $n$ points in the plane
with respect to
a given set of closed disks in the plane.

Finally, we introduced a slightly different cf-coloring function
$C\colon V \to \Naturals$, for which vertices colored with `0' can
not act as uniquely-colored vertices in a hyperedge.
Naturally, one could try to study the bi-criteria optimization
problem, in which there two minimization goals: (a) the number
of colors used, $\max_{v \in V}C(v)$ (minimization of frequency
spectrum use) and (b) the number of vertices with positive colors,
$\cardin{\{v \in V \mid C(v)>0\}}$ (minimization of activated
base stations).

\subsubsection*{Acknowledgments.} We wish to thank Matya Katz and
Asaf Levin for helpful discussions concerning the problems
studied in this paper.

\bibliographystyle{plain}
\bibliography{references}

\end{document}